\documentclass[11pt]{amsart}
\usepackage{amsthm,amssymb,url,hyperref}
\usepackage{fullpage}
\usepackage{stmaryrd}
\usepackage[all]{xy}
\usepackage{color,tikz}

\theoremstyle{plain}
\newtheorem{theorem}{Theorem}[section]

\newtheorem{lemma}[theorem]{Lemma}

\newtheorem{proposition}[theorem]{Proposition}

\theoremstyle{definition}

\newtheorem{example}[theorem]{Example}
\newtheorem{remark}[theorem]{Remark}

\def\fix{\mathrm{fix}}

\def\aut#1{\mathrm{Aut}(#1)}

\def\lmlt#1{\mathrm{LMlt}(#1)}
\def\dis#1{\mathrm{Dis}(#1)}

\def\cjg{\mathrm{Cjg}}
\def\aff{\mathrm{Aff}}

\def\ldiv{\backslash}

\def\Z{\mathbb Z}
\def\F{\mathbb F}

\title{Primitive quandles with alternating displacement group}

\author{Milan Cvr\v cek} 
\address[Cvr\v cek]{Department of Mathematics, Faculty of Science, Humanities and Education, Technical University of Liberec, Czech Republic}
\email{milan.cvrcek@tul.cz}

\author{David Stanovsk\'y}
\address[Stanovsk\'y]{Department of Algebra, Faculty of Mathematics and Physics, Charles University, Prague, Czech Republic}
\email{stanovsk@karlin.mff.cuni.cz}

\begin{document}


\keywords{Simple quandles, primitive quandles, primitive permutation groups.}

\subjclass[2020]{57K12, 20N02, 20B15}
\date{\today}

\begin{abstract}
We classify primitive quandles with alternating displacement group. All of them are conjugation quandles, and the following is a complete list of the underlying conjugacy classes:  
transpositions in $S_n$ for $n=3$ and $n\geq5$; fixpoint-free involutions in $S_n$ for $n\equiv 2\pmod 4$, $n\geq6$; fixpoint-free involutions in $A_n$ for $n\equiv 0\pmod 4$, $n\geq12$; and one exceptional conjugacy class of size 36 in a group of order 720.
\end{abstract}

\maketitle

\section{Introduction}\label{sec:intro}

The theory of quandles was originally motivated by knot theory: the knot quandle is a classifying invariant, and quandle coloring provides computationally feasible invariants \cite{EN}. Quandle coloring motivated various classification projects, including representation and enumeration of small connected quandles \cite{HSV}, isomorphism theorem for affine quandles \cite{Hou}, or classification of finite simple quandles \cite{AG,J-simple}. Both simple and affine quandles are important building stones of general quandles \cite{BS}.

Many properties of quandles are reflected in the permutation group generated by their translations, called the \emph{multiplication group} (or \emph{inner mapping group}, in some papers), and in its subgroup called \emph{displacement group} (see Section \ref{sec:prelim} for precise definitions). A quandle is called \emph{connected} if its multiplication group acts transitively. A connected quandle is affine if and only if its displacement group is abelian. Simple quandles are connected and they are characterized by the fact that every normal subgroup of the multiplication group contains the displacement group. In the finite case, the displacement group is necessarily a power of a finite simple group, which is the main parameter in the classification. 

Connected quandles are homogeneous, because translations are automorphisms. It is therefore natural to ask for classification of quandles with a high degree of symmetry. The problem can be traced back to the beginnings of quandle theory (even before the term ``quandle" was coined by Joyce), see \cite{Nag, Nob}. 

Quandles with a primitive multiplication group, called shortly \emph{primitive quandles}, are simple. Therefore, the classification of finite simple quandles can serve as a basis for classification of finite quandles with a high degree of transitivity.

Finite quandles with doubly transitive multiplication group were classified by Vendramin in \cite{V}: they are precisely the affine quandles $\mathrm{Aff}(\F_q,\alpha)$, where $\alpha$ is a primitive element in $\F_q$. The key step, using deep group theoretic results, is the proof that doubly transitive quandles are affine, thus reducing the problem to a special case solved by Wada \cite{Wada}. Vendramin also explicitly states the problem to classify finite primitive quandles, although implicitly, the problem was discussed already in 1970s \cite{Nag}.

The affine case is easy to solve: an affine quandle is primitive if and only if it is simple. This seems to be a known result, but we could not find an explicit proof in literature, so we include it in Section \ref{sec:affine}. 
Very recently, Bonatto \cite{Bon} classified primitive quandles where the displacement group is a nontrivial power of a finite simple non-abelian group, i.e., $L^t$ with $t>1$, see Theorem \ref{thm:primitive t>1}. In the present paper, we start a systematic approach to the classification of the remaining case, $t=1$. 

Our primary goal is the classification of primitive quandles with alternating displacement group. In this case, the multiplication group is either $A_n$, or $S_n$, or one of the two exceptional groups $A_6\rtimes\Z_2$ resulting from outer automorphisms of $S_6$. The quandles can be desribed as a conjugacy class in their multiplication group. Here we state our main result, the notation will be explained in Section \ref{sec:prelim}.

\begin{theorem} \label{thm:main}
Let $Q$ be a finite quandle with $\mathrm{Dis}(Q)\simeq A_n$, $n\geq3$. Then $Q$ is primitive if and only if one of the following holds:
    \begin{enumerate}        
         \item $Q \simeq \mathrm{Cjg}(S_n,(1,2))$ where $n=3$ or $n\geq5$;
         \item $Q \simeq \mathrm{Cjg}(S_n,(1,2)(3,4)\dots(n-1,n))$ where $n\equiv 2\pmod 4$, $n\geq6$;
         \item $Q \simeq \mathrm{Cjg}(A_n,(1,2)(3,4)\dots(n-1,n))$ where $n\equiv 0\pmod 4$, $n\geq12$;
         \item $Q \simeq \mathrm{Cjg}(G,e)$ where $G=A_6\rtimes\Z_2$ with action by an outer automorphism and
				$e^G$ is the unique conjugacy class with 36 elements ($G=$ {\tt SmallGroup(720,764)} in GAP \cite{GAP}).  
      \end{enumerate}
\end{theorem}

The quandles on the list are pairwise non-isomorphic, with the exception of $\cjg(S_6,(1\ 2))\simeq\cjg(S_6,(1\ 2)(3\ 4)(5\ 6))$.
Indeed, isomorphic quandles have isomorphic multiplication groups, which leaves us with the question of isomorphism between $\mathrm{Cjg}(S_n,(1,2))$ and $\mathrm{Cjg}(S_n,(1,2)(3,4)\dots(n-1,n))$. The two conjugacy classes have different sizes for all $n\neq 6$; while for $n=6$, any outer automorphism of $S_6$ is an isomorphism of the two quandles (cf. \cite[Theorem 7.5]{Rot}).

\begin{remark}
Simple quandles are homogeneous and thus admit a coset representation, $\mathcal Q(G,H,\varphi)=(G/H,*)$ where $xH*yH=x\varphi(x^{-1}y)H$ \cite{J-simple}. On can take, for instance, the minimal representation with $G=\dis Q$, $H=G_e$ (the stabilizer of $e$) and $\varphi=-^{L_e}$ (the automorphism $\alpha\mapsto L_e\alpha L_e^{-1}$) \cite{HSV}. The quandles from Theorem \ref{thm:main} admit the following minimal representations (cf. Lemma \ref{l:LPS}):
    \begin{enumerate}        
         \item[(1)] $Q\simeq\mathcal Q(A_n,(S_2\times S_n)\cap A_n,-^{(1,2)})$;
         \item[(2,3)] $Q\simeq\mathcal Q(A_n,(S_2\wr S_{n/2})\cap A_n,-^{(1,2)\dots(n-1,n)})$;
         \item[(4)] $Q\simeq \mathcal Q(A_6,\mathrm{fix}(\varphi),\varphi)$ where $\varphi$ is the outer automorphism from the semidirect product $A_6\rtimes\Z_2$.
		\end{enumerate}
\end{remark}

\subsection*{The structure of the paper}

In Section \ref{sec:prelim}, we recall basic terminology and notation of quandle theory, we recall the classification of finite simple quandles, define primitive quandles and show that primitive quandles are simple. 
In Section \ref{sec:affine} we prove that affine simple quandles are primitive.
In Section \ref{sec:techniques}, we present two elementary, but general criteria of imprimitivity of conjugation quandles. We also recall a special case of a variant of the O'Nan-Scott theorem that we use to prove primitivity.
In the final Section \ref{sec:proof}, we prove Theorem \ref{thm:main} by case by case analysis of conjugacy classes in $A_n$, $S_n$ and the two exceptional groups $A_6\rtimes\Z_2$.

\section{Preliminaries}\label{sec:prelim}


\subsection{Quandles}\label{ss:prelim_q}
      
A \textit{quandle} is an algebraic structure $Q=(Q,*)$ such that $x*x=x$ for all $x\in Q$, and all \textit{left translations} 
\[ L_{a}:Q \rightarrow Q, \quad x \mapsto a*x\] 
are automorphisms of $(Q,*)$.      

The subgroup of $\aut Q$ generated by the left translations is called the \emph{(left) multiplication group},
\[\mathrm{LMlt}(Q)=\langle L_{x}:x \in Q \rangle.\]
The \textit{displacement group} is its subgroup \[ \mathrm{Dis}(Q)=\langle L_{x}L_{y}^{-1}:x,y \in Q \rangle. \]

\begin{example} 
Let $G$ be a group and $C\subseteq G$ closed with respect to conjugation. The set $C$ posseses a natural quandle operation $x*y=y^x=xyx^{-1}$. This quandle will be denoted by $\cjg(G,C)$ and called a \emph{conjugation quandle} over the group $G$. If $C=G$, we will write shortly $\cjg(G)$. If $C=e^G$, the conjugacy class of $e$ in $G$, we will write shortly $\cjg(G,e)$. 
\end{example}

For every quandle $Q$, we have the \emph{Cayley homomorphism} 
\[ Q\to\cjg(\lmlt Q), \quad x\mapsto L_x. \] 
A quandle is called \emph{faithful} if it is injective.
     
A~quandle $Q$ is called \textit{connected} (or \textit{indecomposable}) if $\mathrm{LMlt}(Q)$ acts transitively on $Q$.
In such a case, all translations are conjugate in $\lmlt Q$, since $L_x^\alpha=L_{\alpha(x)}$. Therefore, the image of the Cayley homomorphism is $\cjg(\lmlt Q, L_e)$, for any $e\in Q$.

A notational remark: $e$ always denotes a fixed element of the underlying set of a quandle $Q$; this includes the case when $Q$ is a subset of a group (e.g., if $Q$ is a conjugation quandle), and in such a case, $e$ is not necessarily the unit element.

To learn more about quandles, we recommend \cite{AG,EN,HSV}.



\subsection{Simple quandles}        

A~quandle $Q$ is called \textit{simple} if $|Q|>1$ and every quandle homomorphism from $Q$ is injective or constant. In other words, if $Q$ has no proper congruence. 
(A technical requirement: congruences are defined with respect to both $*$ and $\ldiv$, where $x\ldiv y=L_x^{-1}(y)$, so that every quotient of a quandle is a quandle, cf. \cite{BS}.)

Every two-element algebraic structure is simple. There is a single two-element quandle, $(\{0,1\},*)$ with $x*y=y$, which is an exception to most statements below.

Simple quandles with at least three elements are faithful (consider the kernel of the Cayley homomorphism) and connected (orbit decomposition is a congruence). Therefore, every simple quandle $Q$ with more than two elements can be represented as a conjugation quandle $\cjg(\lmlt Q,L_e)$.


In \cite{CS-simple}, we propose to classify finite simple quandles according to the structure of their displacement group, due to the following theorem (explicitly formulated in \cite{CS-simple}, but essentially proved by Joyce in \cite{J-simple}).

\begin{theorem}\cite{J-simple}
Let $Q$ be a finite simple quandle, $|Q|>2$. Then there is a finite simple group $L$ and a positive integer $t$ such that $\dis Q\simeq L^t$.
\end{theorem}

Direct powers of finite simple groups are indeed pairwise non-isomorphic, by the Jordan-H\"older theorem.

Let $Q$ be a simple quandle. If $\dis Q$ is an abelian group, then $Q$ is an affine quandle, and we leave the discussion for Section \ref{sec:affine}.
If $\dis Q\simeq L^t$ is a non-abelian group, then $Q$ is a conjugation quandle in a certain extension of $L^t$ modulo its center, as described in the following theorem
(proved by Joyce \cite{J-simple} and refered to as Guralnick's theorem in \cite[Theorem 3.7]{AG}).

\begin{theorem}\cite{AG,J-simple}\label{t:cfsq}
Let $Q$ be a finite simple quandle and $L^t$ a direct power of a finite simple non-abelian group.  Then $\dis Q\simeq L^t$ if and only if 
there is $\varphi\in\aut L$ such that 
\[ Q\simeq\cjg(G_{L,t,\varphi},e), \] where 
	\begin{itemize}
		\item $G_{L,t,\varphi}=L^t\rtimes\langle\varphi^{(t)}\rangle/Z$, where $\varphi^{(t)}(l_1,l_2,...,l_t)=(\varphi(l_t),l_1,l_2,\dots,l_{t-1})$,
		\item $Z$ is the center of $L^t\rtimes\langle\varphi^{(t)}\rangle$,
		\item $e\in G_{L,t,\varphi}$ such that $\langle e^{G_{L,t,\varphi}}\rangle=G_{L,t,\varphi}$. 
	\end{itemize}
Moreover, the group $G_{L,t,\varphi}$ is isomorphic to $\lmlt Q$, and we can choose $e=(1,\varphi^{(t)})Z$.
\end{theorem}

To obtain the groups $G_{L,t,\varphi}$ up to isomorphism, we can choose $\varphi$ up to conjugation in the outer automorphism group of $L^t$ \cite[Theorem 3.7]{AG}:

\begin{proposition}\cite{AG}\label{p:cfsq-iso}   
For $i=1,2$, let $L_i$ be finite simple groups, $t_i$ positive integers and $\varphi_i \in \mathrm{Aut}(L_i^{t_i})$.
Then $G_{L_1,t_1,\varphi_1} \simeq G_{L_2,t_2,\varphi_2}$ if and only if $L_1 \simeq L_2$, $t_1 = t_2$ and $\varphi_1$ and $\varphi_2$ are conjugate as the corresponding elements of $\mathrm{Out}(L_1^{t_1})$. 
\end{proposition}

\subsection{Primitive groups and primitive quandles} 

Let $G$ be a group acting transitively on a set $X$, $|X|>1$.
A \emph{$G$-block} is a subset $B \subseteq X$ such that $gB=B$ or $gB \cap B = \emptyset$ for every $g \in G$. 
A relation $\sim$ on $X$ is called \emph{$G$-invariant} if $x\sim y$ implies $gx\sim gy$ for every $g\in G$.
Indeed, equivalence classes of a $G$-invariant equivalence are $G$-blocks. Conversely, given a $G$-block $B$, we can define a $G$-invariant equivalence by $x\sim y$ iff both $x,y$ belong to $gB$ for some $g\in G$.

The action is called \textit{primitive} if there is no proper $G$-block, or equivalently, no proper $G$-invariant equivalence.
It is well known that a transitive group acts primitively if and only if for each $x \in X$, the stabilizer $G_x$ is a maximal subgroup of $G$. 

A~quandle $Q$ is called \textit{primitive} if $|Q|>1$ and $\mathrm{LMlt}(Q)$ acts primitively on $Q$. This means, there is no proper $\lmlt Q$-invariant equivalence on $Q$, i.e., no proper equivalence invariant with respect to left translations. 

By definition, a congruence of $Q$ is an equivalence invariant with respect to both left and right translations of $Q$. Therefore, primitive quandles are simple, but one can expect that the converse is seldom true (which we confirm explicitly for simple quandles with alternating displacement group). 
We will use the classification of finite simple quandles as the ground for classification of finite primitive quandles. 


The following theorem of Bonatto \cite{Bon} classifies finite primitive quandles over nontrivial powers of simple groups. For reader's convenience, we prove the forward implication in Section \ref{sec:imprim}. 

\begin{theorem}\cite{Bon}\label{thm:primitive t>1}
Let $Q=\cjg(G_{L,t,\varphi},e)$ be a finite simple quandle (in the notation of Theorem \ref{t:cfsq}) such that $G$ is non-abelian and $t>1$. Then $Q$ is primitive if and only if $\varphi=1$ and $t$ is prime.
\end{theorem}

The classification of primitive quandles with $t=1$ is open and the present paper addresses the case of alternating groups.

\section{Affine primitive quandles}\label{sec:affine}

Let $A$ be an abelian group and $f\in\aut A$. Then $\aff(A,f)=(A,*)$ under the operation \[ a*b=(1-f)(a)+f(b) \] is a quandle. Every quandle isomorphic to some $\aff(A,f)$ is called \emph{affine}.
It is not difficult to calculate that the displacement group is isomorphic to $\mathrm{Im}(1-f)\leq A$. Conversely, every connected quandle with an abelian displacement group is affine \cite[Theorem 7.3]{HSV}.

For affine simple quandles, there is a more convenient description than the one provided by Theorem \ref{t:cfsq}.

\begin{theorem}\cite{AG}
Let $Q$ be a finite simple quandle. The following are equivalent:
\begin{enumerate}
	\item $\dis Q$ is abelian;
	\item there is a prime $p$, positive integer $t$ and a linear operator $f$ acting irreducibly on the space $\Z_p^t$ such that $Q\simeq\aff(\Z_p^t,f)$.
\end{enumerate}
\end{theorem}

In particular, condition (2) implies that $1-f$ is an automorphism, $Q$ is connected (in fact, latin), and thus $\dis Q\simeq\Z_p^t$.

It is not difficult to prove that all finite affine simple quandles are primitive. This seems to be a known result, but we could not find an explicit proof in literature, so we include it here.

\begin{proposition}
A finite affine quandle is primitive if and only if it is simple.
\end{proposition}

\begin{proof}
The forward implication is true for every quandle, we prove the converse.

Let $Q=\aff(\Z_p^t,f)$ where $f$ acts irreducibly on $\Z_p^t$. Since $Q$ is connected, $1-f$ is onto, and since $Q$ is finite, it is bijective. 
Denote $a/b$ the unique $x$ such that $a=x*b$ (explicitly, $a/b=(1-f)^{-1}(a-f(b))$) and observe that $a+b=(a/0)*(0\ldiv b)$. 

Consider a proper $\lmlt Q$-invariant equivalence $\sim$ on $Q$. 
We will prove that $H=\{a:a\sim 0\}$ is an $f$-invariant subspace of $\Z_p^t$:
if $a,b\sim 0$, then $0\ldiv b\sim 0\ldiv 0=0$, hence $a+b=(a/0)*(0 \ldiv b)\sim (a/0)*0=a\sim 0$, so $H$ is a subspace; and we also have $f(a)=0*a\sim 0*0=0$, so $H$ is $f$-invariant.
Now observe that $H\neq\{0\}$, since $a\sim b$ implies $(0/b)*a\sim(0/b)*b=0$ and $(0/b)*a=0$ if and only if $a=b$. 
But $f$ acts irreducibly, so $H=\Z_p^t$ and thus $\sim$ is the full equivalence.
\end{proof}

\section{Proof techniques}\label{sec:techniques}

\subsection{Sources of imprimitivity}\label{sec:imprim}

We present two general negative results for conjugation quandles. 
Let $G$ be a group and $e\in G$ such that $\langle e^G\rangle=G$, and denote $Q=\cjg(G,e)$. Then $\lmlt Q\simeq\mathrm{Inn}(G)$, hence the action of $\lmlt Q$ on $Q$ is equivalent to the action of $G$ on $e^G$ by conjugation (with kernel $Z(G)$).

        


\begin{lemma} \label{mocnina} 
Let $G$ be a group and $e\in G$ such that $\langle e^G\rangle=G$. If $\{e\}\subsetneq \langle e\rangle\cap e^G\subsetneq e^G$, then $\cjg(G,e)$ is imprimitive.
\end{lemma} 

\begin{proof} 
Consider the relation $\sim$ on the set $e^G$ defined by $g\sim h$ if and only if $h=g^k$ for some $k\in\Z$. The relation is symmetric: $g,h$ have the same order $n$, hence $k$ is invertible modulo $n$ (if $n=\infty$, the case is clear). It follows that $\sim$ is an equivalence. Clearly, it is $G$-invariant. The assumption implies that it is proper.
\end{proof}

We will often apply this lemma in situation when $e\neq e^{-1}\in e^G$, and ocassionally in its full strength.

In the following lemma, let $\fix(g)$ denote the set of fixed points of a permutation $g$.

\begin{lemma} \label{fix} 
Let $G$ be a group and $e\in G$ such that $\langle e^G\rangle=G$. If there exists $e\neq f\in e^G$ such that $\fix(e)=\fix(f)\neq\emptyset$, then $\cjg(G,e)$ is imprimitive.
\end{lemma} 
        
\begin{proof} 
Consider the relation $\sim$ on the set $e^G$ defined by $g\sim h$ if and only if $\fix(g)=\fix(h)$. Clearly, this is a $G$-invariant equivalence (note that $G$ acts on $e^G$ by conjugation). We have $e\sim f$, hence $\sim$ is nontrivial. Since $e$ has at least one fixed point, using transitivity, we can always find $g\in G$ such that $e$ and $e^g$ have different sets of fixed points, hence $\sim$ is not the full equivalence. 
\end{proof}   
         
We expect that the criteria of the present section are useful for a wide range of simple quandles, over various types of simple groups. Quandles over alternating groups, discussed in the next section, serve as a proof of concept.

A different approach to proving imprimitivity can be outlined as follows. Assume that a group $G=N\rtimes\langle\psi\rangle$, $\psi\in\aut N$, acts on the conjugacy class $(1,\psi)^G$ by conjugation.
Since $(1,\psi)^{(a,\psi^i)}=(a\psi(a)^{-1},\varphi)$, the stabilizer is 
$G_{(1,\psi)}=\{(a,\psi^i):\psi(a)=a\}=\fix(\psi)\rtimes\langle\psi\rangle$. In turn, we see that $Z(G)\leq \fix(\psi)\rtimes\langle\psi\rangle$.
In some cases, it is easy to show that the stabilizer is not a maximal subgroup.

\begin{proof}[Proof of Theorem \ref{thm:primitive t>1}, part $(\Rightarrow)$]
Let $G=G_{L,t,\varphi}=L^t\rtimes\langle\psi\rangle/Z(L^t\rtimes\langle\psi\rangle)$ where $\psi=\varphi^{(t)}$. Denote $D_X=\{(a,\ldots,a):a\in X\}$. 
It follows from the discussion above that $Z(L^t\rtimes\langle\psi\rangle)\leq D_{\fix(\varphi)}\rtimes\langle\psi\rangle$, and thus
$G_{(1,\psi)Z}=D_{\fix(\varphi)}\rtimes\langle\psi\rangle/Z$. If $\varphi\neq1$, then $$G_{(1,\psi)Z}<D_L\rtimes\langle\psi\rangle/Z<G,$$ hence $G$ acts imprimitively. If $\varphi=1$ and $s\mid t$ is a proper divisor, then 
$$G_{(1,\psi)Z}=D_L\rtimes\langle\psi\rangle/Z<\{(a_1,\ldots,a_s,\ldots,a_1,\ldots,a_s):a_1,\ldots,a_s\in L\}\rtimes\langle\psi^s\rangle/Z<G,$$ hence $G$ acts imprimitively.
\end{proof}

\subsection{O'Nan-Scott for primitivity}

To prove primitivity, we use a variant of the O'Nan-Scott theorem by Liebeck, Praeger, Saxl \cite{LPS}, which describes maximal subgroups of the groups $A_n$ and $S_n$. Only two types of maximal subgroups appear here: the imprimitive ($S_2\times S_m$) and the intransitive ($S_2\wr S_m$) ones. Specifically, we need the following part of the theorem.

\begin{lemma}\cite{LPS}\label{l:LPS}
Let $n>2$.
\begin{enumerate}
	\item $S_2\times S_{n-2}$ is a maximal subgroup of $S_n$.
\end{enumerate}
If $n$ is even, then 
\begin{enumerate}
	\item[(2)] $S_2\wr S_{n/2}$ is a maximal subgroup of $S_n$;
	\item[(3)] if $n\neq 8$, then $(S_2\wr S_{n/2})\cap A_n$ is a maximal subgroup of $A_n$.
\end{enumerate}
For $n=8$, $(S_2\wr S_{n/2})\cap A_n < \mathrm{AGL}_3(2) < A_8$.
\end{lemma}

Note that our proof of Theorem \ref{thm:main} relies only on the easier part of the O'Nan-Scott theorem, stating that certain subgroups are maximal. 

We also note that for small groups, primitivity can be decided computationally, for instance, in the GAP system. We will utilize purely computational approach in case of the groups $A_6\rtimes\Z_2$.

\section{Primitive quandles with alternating displacement group}\label{sec:proof}

Assume that $Q$ is a simple quandle with $\dis Q\simeq A_n$. If $n\leq 3$, then $A_n$ is abelian, $Q$ is affine, and we refer to Section \ref{sec:affine}. The group $A_4$ is not simple. So, in what follows, we assume that $n\geq5$. 

Recall that $\mathrm{Out}(A_n)\simeq\Z_2$ for all $n\neq 6$ and $\mathrm{Out}(A_6)\simeq\Z_2\times\Z_2$. We start by discussion of the structure of the multiplication group.

\begin{lemma}\label{l:lmlt}
Let $Q$ be a simple quandle with $\dis Q\simeq A_n$, $n\geq 5$. Then $\lmlt Q$ is isomorphic to $A_n$ or $S_n$ whenever $n\neq 6$. For $n=6$, $\lmlt Q$ is isomorphic to $A_6$, or to one of $A_6\rtimes\Z_2$ where $\Z_2$ acts as an outer automorphism (there are three pairwise non-isomorphic groups of this sort, one of them is $S_6$, the other two are {\tt SmallGroup(720,$x$)} with $x\in\{764,765\}$ in the GAP notation).
\end{lemma}

\begin{proof}
According to Theorem \ref{t:cfsq}, $Q$ is isomorphic to $\mathrm{Cjg}(G_{A_n,1,\varphi},e)$ where $\lmlt Q\simeq G_{A_n,1,\varphi}=(A_n\rtimes\langle\varphi\rangle)/Z$.
Moreover, Proposition \ref{p:cfsq-iso} asserts that $G_{A_n,1,\varphi}\simeq G_{A_n,1,\psi}$ if and only if $\varphi,\psi$ are conjugate in $\mathrm{Out}(A_n)$, i.e., if and only if $\varphi\psi^{-1}\in\mathrm{Inn}(A_n)$.

So, for $n\neq 6$, there are two options: the inner automorphism can be represented by $\varphi=id$ and we obtain $G_{A_n,1,\varphi}=A_n$; the outer automorphism can be represented by the automorphism $\phi_{(1,2)}$ (conjugation by a transposition) and we obtain $G_{A_n,1,\varphi}=S_n$. For $n=6$, we have two more options, both can be represented by outer automorphisms of order two; one can check in GAP that both semidirect products are centerless.
\end{proof}

       

Consequently, using Cayley representation, $Q$ is isomorphic to a conjugation quandle over the group $A_n$ or $S_n$ or one of the exceptional groups for $n=6$.

\begin{lemma}\label{l:case S_n}
Let $Q=\mathrm{Cjg}(S_n,\pi)$ be a simple quandle, $n\geq 5$. Then $Q$ is primitive if and only if $\pi$ is a transposition, or $n\equiv2\pmod4$ and $\pi$ consists of $n/2$ transpositions.
\end{lemma}

\begin{proof}
First assume that $\mathrm{ord}(\pi)>2$. Since $\pi\neq \pi^{-1}\in \pi^{S_n}$, we can apply Lemma \ref{mocnina} and obtain that $\cjg(S_n,\pi)$ is imprimitive.

Now assume that $\mathrm{ord}(\pi)=2$, $\pi$ is a non-transposition, $\pi$ has a fixpoint. Then we can apply Lemma \ref{fix} and and obtain that $\cjg(S_n,\pi)$ is imprimitive.

The remaining cases are transpositions and compositions of $n/2$ disjoint transpositions. 

If $\pi$ is a transposition, then $\langle\pi^{S_n}\rangle=S_n$, the stabilizer of $\pi$ acts as $S_2\times S_{n-2}$, which is a maximal subgroup by Lemma \ref{l:LPS}. Therefore, $\cjg(S_n,\pi)$ is primitive.

Let $\pi$ be a composition of $n/2$ transpositions. If $4\mid n$, then $\langle\pi^{S_n}\rangle=A_n$, so $Q$ is not simple. Else, we obtain primitivity from Lemma \ref{l:LPS}, since the stabilizer acts as $S_2\wr S_{n/2}$.
\end{proof}

\begin{lemma}\label{l:case A_n}
Let $Q=\mathrm{Cjg}(A_n,\pi)$ be a simple quandle, $n\geq 5$. Then $Q$ is primitive if and only if $4\mid n\geq 12$ and $\pi$ consists of $n/2$ transpositions.
\end{lemma}

\begin{proof}
It is well known, and easy to check, that $\pi^{A_n}=\pi^{S_n}$ if and only if $\pi$ contains a cycle of even length or $\pi$ contains two cycles of the same odd length.

First assume that $\pi^{A_n}=\pi^{S_n}$. 
If $\mathrm{ord}(\pi)>2$, then $\pi\neq \pi^{-1}\in \pi^{A_n}$, and $\cjg(A_n,\pi)$ is imprimitive by Lemma \ref{mocnina}. 
If $\mathrm{ord}(\pi)=2$ and $\pi$ has a fixpoint, then $\cjg(A_n,\pi)$ is imprimitive by Lemma \ref{fix}.
Else, $\pi$ is a composition of $n/2$ disjoint transpositions (thus $4\mid n$), and the case is solved by Lemma \ref{l:LPS}.

In the remaining cases, $\pi$ is a composition of disjoint odd cycles of distinct lengths. Then there is $\sigma\in S_n$ such that $\sigma\pi\sigma^{-1}=\pi^2$. Consequently, $\sigma^2\pi\sigma^{-2}=\pi^4$, hence  $\pi\neq \pi^4\in \pi^{A_n}$, and $\cjg(A_n,\pi)$ is imprimitive by Lemma \ref{mocnina}.  (The inequality follows from the fact that $n>4$.)
\end{proof}

\begin{lemma}\label{l:case A_6}
Let $G=A_6\rtimes\Z_2$ where $\Z_2$ acts as an outer automorphism. The number of primitive quandles of the form $\mathrm{Cjg}(G,e)$ is
\begin{itemize}
	\item two when $G=S_6$, their orders are 15;
	\item one when $G=$ {\tt SmallGroup(720,764)}, its order is 36;
	\item zero when $G=$ {\tt SmallGroup(720,765)}.
\end{itemize}
\end{lemma}

\begin{proof}
The statement can be easily verified on a computer in GAP, using the following code for each of the three groups:
\begin{center}
{\tt Filtered(ConjugacyClasses(G),C->(Group(Elements(C))=G and IsPrimitive(G,C)));}
\end{center}
\end{proof}

We remark that the two quandles over $S_6$, corresponding to the conjugacy classes $(1,2)^{S_6}$ and $(1,2)(3,4)(5,6)^{S_6}$, are isomorphic, by any outer automorphism of $S_6$.

The four lemmas complete the proof of our main theorem.

\begin{proof}[Proof of Theorem \ref{thm:main}]
Primitive quandles are simple, hence assume that $Q$ is a finite simple quandle with $\dis Q\simeq A_n$. Lemma \ref{l:lmlt} describes possible multiplication groups $G$, and Lemmas \ref{l:case S_n}, \ref{l:case A_n} and \ref{l:case A_6} classify conjugation quandles over $G$ that are primitive. 
\end{proof}


\end{document}